\documentclass[11pt]{amsart}
\usepackage{amsmath}
\usepackage{amsfonts}
\usepackage{amsthm}
\usepackage{amssymb}
\usepackage{mathrsfs}
\usepackage{latexsym}
\usepackage{verbatim}
\usepackage{a4wide}
\usepackage{enumerate}
\usepackage{lipsum}

\numberwithin{equation}{section}

\renewcommand{\thetheoremName}

\def\new{\mathrm{new}}

\def\Z{\mathbf{Z}}

\def\OL{\mathcal{O}}

\def\Q{\mathbf{Q}}

\def\R{\mathbf{R}}

\def\C{\mathbf{C}}

\def\a{\mathfrak{a}}
\def\m{\mathfrak{m}}
\def\p{\mathfrak{p}}
\def\d{\mathfrak{d}}

\def\GL{\mathrm{GL}}

\def\Gal{\mathrm{Gal}}

\newtheorem{theorem}{Theorem}[section]
\newtheorem*{theorem*}{Main Theorem}

\newtheorem{prop}[theorem]{Proposition}

\newtheorem{remark}[theorem]{Remark}

\newtheorem{alg}{Algorithm}

\def\H{\mathbf{H}}
\def\n{\mathfrak{n}}
\def\q{\mathfrak{q}}

\begin{document}

\title{There Exist Non-CM Hilbert modular forms of partial weight one}
\author{Richard A. Moy}

\thanks{Corresponding Author: Richard Moy, Northwestern University,
2033 Sheridan Road, Evanston, IL, United States. Email: ramoy88@math.northwestern.edu.}

\address{Richard Moy, Department of Mathematics, Northwestern University, 2033 Sheridan Road, Evanston, IL 60208, United States}
\email{ramoy88@math.northwestern.edu}

\author{Joel Specter}
\address{Joel Specter, Department of Mathematics, Northwestern University, 2033 Sheridan Road, Evanston, IL 60208, United States}
\email{jspecter@math.northwestern.edu}



\begin{abstract}
In this note, we prove that there exists a classical Hilbert modular cusp form over $\Q(\sqrt{5})$ of partial weight one which does \textit{not} arise from the induction of a Gr\"ossencharacter from a CM extension of $\Q(\sqrt{5}).$
\end{abstract}

\maketitle

\section{Introduction}\label{sect:intro}
It is a well-established ``folklore question''\footnote{We originally learnt of this problem through Fred Diamond. In conversations with Kevin Buzzard, Don Blasius, and Fred Diamond, it became clear that the question of whether such forms existed was apparent to the authors of~\cite{BR} in the '80s (and may well have occurred to others before then). The question gained some urgency with the advent of Fraser Jarvis' construction of Galois representations for partial weight one forms \cite{J} in the mid-'90s, since, if the only such forms were CM, then \cite{J} would be a trivial consequence of Class field theory. We have heard several reports of the question being raised again at this time. In light of these stories, we feel safe in calling the problem a ``well-known folklore question.''} whether there exists a totally real field $F$ and a classical Hilbert modular form $f$ of partial weight one which does \textit{not} arise from the induction of a Gr\"ossencharacter from some CM extension of $F$. This note answers the question in the affirmative (see Theorem~\ref{theorem:main}). If, in addition, all the weights of $f$ have the same parity, then, assuming local-global compatibility, there exists a compatible
family of representations
$(L,\{\rho_{\lambda}\})$ with the following intriguing property:
\begin{quote}
Let $\ell$ be a prime in $\OL_F$ not dividing the level of $f$ and totally split in $F$. If $\lambda$ is a prime in $\OL_L$ above $\ell$, then the corresponding representation,
$$\rho_{\lambda}: G_F \rightarrow \GL_2(\OL_{\lambda})$$
will be geometric, have Zariski dense image, and yet be unramified for at least one $v|\ell$.
\end{quote}
Many cases of local-global compatibility are now known \cite{L}. Although such a  beast seems somewhat peculiar, there is no obvious \emph{a priori} reasons why it should not exist. On the other hand, there does not seem to be any obvious way (even conjecturally) to produce such a modular form, either by automorphic or motivic methods. Hence, to answer the question, we must find such a form, which we do. Although (in principle) the method of computation used in this paper applies to general totally real fields, we shall restrict to real quadratic fields $F$ with narrow class number one for convenience. Indeed,  all of our computations took place with Hilbert modular forms for the field $F = \Q(\sqrt{5})$.

\subsection{The Computation}\label{sect:computation} Our search for partial weight one Hilbert modular forms is premised on the philosophy that finite dimensional spaces of meromorphic modular forms which are stable under the action of the Hecke algebra ought to be modular. In the case of classical modular forms, this idea has been formalized by George Schaeffer. Let $V$ be a finite dimensional space of meromorphic modular forms on $\Gamma_0(N)$ of weight ${k}$ and nebentypus $\chi$ which are holomorphic at infinity. In his thesis \cite{S}, Schaeffer proves that if $V$ is stable under the action of a Hecke operator $T_p$ for $p\nmid N,$ then $V \subseteq M_{{k}}(\Gamma_0(N), \chi, \C)$. As a corollary, one observes that for any such $V$ containing $M_{{k}}(\Gamma_0(N), \chi, \C),$ the chain (for $p\not|N$)
$$
V \supseteq V \cap T_pV \supseteq V \cap T_pV \cap T_p^2V \supseteq ...
$$
stabilizes to $M_k(\Gamma_0(N), \chi, \C)$ in less than $\dim_{\C} V$ steps \cite{S,S2}.
\medskip

Schaeffer's principal application of this theorem is the effective computation of the space $M_{{1}}(\Gamma_0(N), \chi, \C)$ of classical weight one modular forms. Suppose one wishes to compute this space. To begin, simply take any Eisenstein series $E \in M_{{1}}(\Gamma_0(N), \chi^{-1}, \C)$ and let $V$ be the space of ratios of forms in $M_2(\Gamma_0(N), \C)$ and $E.$ Then $V\supseteq M_{{1}}(\Gamma_0(N), \chi, \C).$ It then suffices to compute the intersection of $V$ with its Hecke translates. One can reduce this computation to one in linear algebra by passing to Fourier expansions. The Fourier expansions of forms in $M_2(\Gamma_0(N), \C)$ are easily calculated to any bound via modular symbols and the Fourier expansion of $E$ has a simple formula. Hence, the Fourier expansion of any form in $V$ is easily calculated to any bound. The operator $T_p$ acts on Fourier expansions formally via a well known formula. What makes the method effective is that it requires only an explict finite number of Fourier coefficients for a basis of the space $V$ to calculate $M_{{1}}(\Gamma_0(N), \chi, \C).$ The number of coefficients required is determined by the Sturm bounds.
\medskip

Schaeffer's method generalizes nicely to the case of Hilbert modular forms. Let $\mathfrak{n}$ be a modulus of $\Q(\sqrt{5})$ and $\chi$ a ray class character of conductor $\mathfrak{n}$; we are interested in calculating the space $S_{[m,1]}(\Gamma_0(\mathfrak{n}), \chi, \C)$ of Hilbert cusp forms of partial weight one. As in the case of classical modular forms, there exists an Eisenstein series $E_{1,{\chi^{-1}}} \in M_{[1,1]}(\Gamma_0(\mathfrak{n}), \chi^{-1}, \C)$ and one can consider the space $V$ of ratios with numerators in $S_{[m+1,2]}(\Gamma_0(\mathfrak{n}),  \C)$ and denominator $E_{1,{\chi^{-1}}}.$ This is a finite dimensional space of meromorphic forms which contains $S_{[m,1]}(\Gamma_0(\mathfrak{n}), \chi, \C)$ as its maximal holomorphic subspace. Assuming $\mathfrak{n}$ is square free, one can use Dembele's algorithm \cite{D} as implemented in {\tt magma\rm} \cite{BCP}, to produce the Fourier expansions of a basis for the space $S_{[m+1,2]}(\Gamma_0(\mathfrak{n}), \Q(\sqrt{5}))$ to any desired degree of accuracy. The Fourier expansion of $E_{1,{\chi^{-1}}}$ is given by an explicit formula. Hence, the Fourier expansion of the meromorphic forms in $V$ can be calculated to any desired degree of accuracy. For a prime $\mathfrak{p}$ of $\mathcal{O}_F,$ the Hecke operator $T_{\mathfrak{p}}$ acts on the Fourier expansions of the meromorphic forms in $V$ formally via an explicit formula. So, as in the case of classical forms, one may hope to calculate the $T_{\mathfrak{p}}$ stable subspace of $V$ via techniques in linear algebra.
\medskip

Unfortunately, this direct generalization of Schaeffer's method is impractical from a computational prospective. In comparison with the case of classical modular forms, the number of Fourier coefficients needed to prove equality of two modular forms and the amount of computation needed to calculate those Fourier coefficients is much greater. For this reason, we structure our search method so that it requires as few Fourier coefficients as possible.
\medskip

For the details of our search, we refer the reader to Section \ref{sect:algorithm}. But the idea is as follows; we calculate the Fourier expansions of the forms in $S_{[m+1,2]}(\Gamma_0(\mathfrak{n}),1) / E_{1,\chi^{-1}}$ truncated to some chosen bound. We calculate the intersection of these spaces of truncated formal Fourier expansions using linear algebra. If the dimension of the intersection coincides with the dimension of the subspace of forms with complex multiplication (CM), then every form in $S_{[n,1]}(\Gamma_0(\mathfrak{n}), \chi, \Q(\sqrt{5}))$ has CM. (Using class field theory, we can compute the dimension of the CM subspace in advance.) In practice, using a modest bound, we were able to restrict the existence of a non-CM weight one Hilbert modular form of small level and norm to a handful of candidate spaces where the intersection is larger than expected. One can then check if a form $f \in V$ in such a candidate space is holomorphic by checking if there exists a form $g \in S_{[dn,d]}(\Gamma_0(\mathfrak{n}), \chi^d, \Q(\sqrt{5}))$ such that the $f^d=g$. Our search yielded the existence of a nonparallel weight one Hilbert modular form without CM.

\begin{theorem*} \label{theorem:main} Let $\n = (14) \subset \OL_{\Q(\sqrt{5})}$ and let $\chi$ be the degree $6$ ray class character of conductor $(7)\cdot\infty_1\infty_2$ such that
$\chi(2) =  \frac{-1 + \sqrt{-3}}{2}$. The space of cusp forms
$S_{[5,1]}(\Gamma_0(\mathfrak{n}),\chi,\C)$ is $2$-dimensional, and has a basis with coefficients in $H=\Q(\sqrt{5},\chi)$. This space has a basis over $L=H(\sqrt{-19})$ consisting of two conjugate eigenforms, neither of which admit complex multiplication.
\end{theorem*}

\begin{remark} \emph{Let $\pi$ be the automorphic representation of $\GL_2(\mathbb{A}^{\infty}_F)$ associated to either of these newforms.
Since the character~$\chi$ has conductor prime to~$2$ and the level at~$2$ is $\Gamma_0(2)$,   the local component~$\pi_2$ is Steinberg (up to an unramified quadratic twist).
In particular, this implies that local-global compatibility results of~\cite{L,N} could not be proved directly using congruence methods to higher
weight, which would only be sufficient for proving compatibility up to~$N$ semi-simplification. }
\end{remark}


\section{Hilbert Modular Forms}

In this section, we state some basic definitions and results on classical Hilbert modular forms. Let $F$ be a real quadratic field of narrow class number one. We fix an ordering on the two embeddings of $F$ into $\R$ and denote, for $a \in F,$ the image of $a$ under the $i$-th embedding by $a_i.$ We say an element $a\in F$ is totally positive if $a_i>0$ for all $i$ and denote the ring of all such elements by $\OL_F^{+}$. Similarly, we have two natural embeddings of the matrix ring $\textup{M}_2(F)$ into the matrix ring $\textup{M}_2(\R)$. If $\gamma\in\textup{M}_2(F)$, let $\gamma_1$ and $\gamma_2$ denote the image of $\gamma$ under the $i$-th embedding. Let $\mathfrak{d}_{F}=(\delta)$ be the different of $F/\Q$ where $\delta\in\OL_F^+$. For an integral ideal $\mathfrak{n}$ of $F,$ we define

$$\Gamma_0(\mathfrak{n}) := \left\{ \left(
\begin{matrix}
a & b\\
c & d\\
\end{matrix}\right)\in \textup{GL}_2^+(F): a,d\in\OL_F,\ \  c\in\mathfrak{n}\mathfrak{d},\ \ b\in\d^{-1},\ \ ad-bc\in\OL_F^\times \right\}$$ where $\textup{GL}_2^+(F)$ is the subgroup of $\textup{GL}_2(F)$ composed of matrices with totally positive determinant. If $\H$ is the complex upper half-plane, the group $\Gamma_0(\mathfrak{n})$ acts on $\H \times \H$ via fractional linear transformations by the rule

$$\left(\begin{matrix}
a & b\\
c & d\\
\end{matrix} \right).(z_1,z_2) = \left(\frac{a_1z_1 + b_1}{c_1z_1 + d_1}, \frac{a_2z_2 + b_2}{c_2z_2 + d_2} \right).$$

Let $\underline{k} := [k_1, k_2]$ be an ordered pair of nonnegative integers. For $\gamma = $$\left(\begin{matrix}
a & b\\
c & d\\
\end{matrix}\right) \in \textup{GL}_2^+(F)$ and $z \in \H \times \H$ set $$j(\gamma,z)^{\underline{k}} := \det(\gamma_1)^{\frac{-k_1}{2}}\det(\gamma_2)^{\frac{-k_2}{2}}(c_1z_1 + d_1)^{k_1}(c_2z_2 + d_2)^{k_2}.$$

If $f:\H\times\H\rightarrow\C$ and $\gamma \in \textup{GL}_2^+(F),$ we write $f|_{\gamma}$ to mean the function $f|_{\gamma}: \H \times \H \rightarrow \C$ given by
$$
f|_{\gamma}(z) = j(\gamma,z)^{-\underline{k}} f(\gamma z).
$$

Consider a numerical character $\chi:(\mathcal{O}_F/\mathfrak{n})^{\times} \rightarrow \C^{\times}$ which satisfies $\chi(u) = \left(\frac{u_1}{|u_1|}\right)^{-k_1}\left(\frac{u_2}{|u_2|}\right)^{-k_2}$ for all $u\in \mathcal{O}_F^{\times}.$ A \emph{Hilbert modular form} of weight $\underline{k},$ level $\mathfrak{n}$,  and character $\chi$ is a holomorphic function $f\colon \H \times \H\rightarrow\C$ such that for all $\gamma \in \Gamma_0(\mathfrak{n})$,
\begin{equation}\label{eq_transfrule}
f|_{\gamma}(z)=\chi(d)f(z).
\end{equation}

We denote the $\C$-vector space of all such functions by $M_{\underline{k}} (\Gamma_0(\mathfrak{n}), \chi, \C)$ and by $M_{\underline{k}}(\Gamma_0(\mathfrak{n}),\C)$ when $\chi$ is the trivial character. As in the case of classical modular forms, we can compute Fourier expansions of Hilbert modular forms.

\subsection{Fourier Expansions}\label{sect:qexp}

If $f \in M_{\underline{k}}(\Gamma_0(\mathfrak{n}),\chi, \C),$ then for all $d \in \mathfrak{d}^{-1}_F$
$$f(z) = f(z + d)$$ by the transformation rule \eqref{eq_transfrule} since $\left(\begin{smallmatrix} 1& d\\ 0 & 1\\ \end{smallmatrix}\right)\in\Gamma_0(\mathfrak{n})$. It follows from Fourier analysis that the form $f$ is given by the series
$$
f(z) = \sum_{\alpha \in \OL_F} c_{\alpha}(f) e^{2\pi i( \alpha_1z_1 + \alpha_2z_2)}
$$
in a neighborhood of the cusp $(\infty, \infty).$ The Koecher Principle \cite[\textsection 1]{G} states that $c_{\alpha}(f) = 0$ unless $\alpha$ is totally positive or $\alpha=0.$ If the constant term of the Fourier expansion of $f|_\gamma$ is zero for all $\gamma\in\textup{GL}_2^+(F)$, then we call $f$ a \textit{cusp form} and denote the space of such forms $S_{\underline{k}}(\Gamma_0(\mathfrak{n}), \chi, \C).$ We denote the space of cusp forms of level $\mathfrak{n}$, weight $\underline k$, and trivial character by $S_{\underline k}(\Gamma_0(\mathfrak{n}),\C)$.

Besides the Koecher Principle, the Fourier expansions of Hilbert modular forms have additional structure. Let $f \in S_{\underline{k}}(\Gamma_0(\mathfrak{n}),\chi, \C).$ For any totally positive unit $\eta$ in $\mathcal{O}_F,$ one can check that the coefficient $c_{\alpha}(f)$ satisfies the identity:
\begin{equation}\label{eq_coeff1}
c_{\eta \alpha}(f) = \eta_1^{k_1/2}\cdot\eta_2^{k_2/2}\cdot c_{\alpha}(f)= \eta_2^{(k_2-k_1)/2}\cdot c_{\alpha}(f)
\end{equation}
by using the transformation rule \eqref{eq_transfrule} with $\left(\begin{smallmatrix}\eta&0\\ 0&1\\ \end{smallmatrix}\right)\in\Gamma_0(\mathfrak{n})$ and equating Fourier expansions. If desired, we can create a formal Fourier expansion indexed over the ideals of $F$ rather than indexed over elements of $\OL_F$. In particular, for an ideal $\a = (\alpha)$, we can set
\begin{equation}\label{eq_coeff2}
c(\a,f):=N(\a)^{(k_1-k_2)/2}\cdot c_{\alpha}(f)/\alpha_1^{(k_1-k_2)/2} = c_{\alpha}(f) \cdot {\alpha_2}^{(k_1-k_2)/2},
\end{equation}
and one can easily check that this is independent of the choice of totally positive generator $\alpha$ of $\a$ by using \eqref{eq_coeff1} above.

\subsection{Hecke Operators}\label{sect:hecke}

For an integral ideal $\mathfrak{n}$ of $\OL_F$, let $$\Gamma_1(\mathfrak{n}) = \left\{ \left(
\begin{matrix}
a & b\\
c & d\\
\end{matrix}\right)\in \textup{GL}_2^+(F): a\in\OL_F,\ \  b\in\mathfrak{d}^{-1},\ \ c\in\mathfrak{n}\mathfrak{d},\ \ d-1 \in  \mathfrak{n}\right\}.$$  If $\mathfrak{q}$ is an integral ideal of $\mathcal{O}_F$, we may choose a totally positive generator $\pi$ of $\mathfrak{q}$ and write the disjoint union

$$
\Gamma_1(\mathfrak{n})\left(
\begin{matrix}
1 & 0\\
0 & \pi \\
\end{matrix}\right)
\Gamma_1(\mathfrak{n})=\coprod_j\Gamma_1(\mathfrak{n})\gamma_j
$$
where the $\gamma_j$ are a finite set of right coset representatives. We define the \emph{$\mathfrak{q}^{th}$ Hecke operator} to be
\begin{equation}\label{eq_hecke}
T_{\mathfrak{q}} f := \sum_j f|_{\gamma_j}.
\end{equation}

If $\mathfrak{q}=(\pi)$ is a prime ideal relatively prime to $\mathfrak{n}$, then we may choose the following coset representatives for the $\gamma_j$:
$$
\gamma_{\beta}:=\left(\begin{matrix}1 & \epsilon \delta^{-1}\\ 0 & \pi\\ \end{matrix} \right)\ \ \ \textup{and}\ \ \ \gamma_{\infty}:=\left(\begin{matrix}\alpha & \beta \delta^{-1} \\ \delta \nu&\pi \\ \end{matrix} \right)\left(\begin{matrix}\pi &0 \\ 0 &1\\ \end{matrix}\right)
$$
where $\epsilon$ runs through a complete set of representatives for $\OL_F\slash \mathfrak{n}$, $\delta$ is a totally positive generator for the different $\mathfrak{d}$, $\nu$ is a totally positive generator for $\mathfrak{n}$, and $\alpha,\beta\in\OL_F$ such that $\alpha\pi-\nu \beta=1$. If we normalize our Hecke operator by multiplying it by the constant $\pi_1^{k_1/2-1}\pi_2^{k_2/2-1}$, then it has the following effect on the Fourier expansion of a modular form $f\in M_{\underline k}(\Gamma_0(\mathfrak{n}),\chi,\C)$:
$$c_{\alpha}(T_{\mathfrak{q}} f) = c_{\alpha \pi} + \pi_1^{k_1-1} {\pi_2}^{k_2-1} \chi(\q) c_{\alpha/\pi} = c_{\alpha \pi} + \pi_2^{k_2-k_1} N(\q)^{k_1-1} \chi(\q) c_{\alpha/\pi}$$
where $N(\mathfrak{q})$ denotes the numerical norm of the ideal $\mathfrak{q}$. On the other hand, if $\q$ is prime and exactly divides $\n$, then
$$c_{\alpha}(T_{\pi} f) = c_{\alpha \pi}.$$

\subsection{Basis For $S_{\underline k}(\Gamma_0(\mathfrak{n}),\chi,\C)$}\label{sect:basis}
In general, there will not be a basis of eigenforms for $S_{\underline{k}}(\Gamma_0(\n),\chi,\C)$. Rather,  there will be a \emph{new-space}
$S^{\mathrm{new}}_{\underline{k}}(\Gamma_0(\n),\chi,\C)$ which
will be generated by eigenforms which we now describe.
\medskip

Let $\m$ be a divisor of $\n$, and let $\mathfrak{b}$ be a divisor of $\n/\m$. Then
there is a map
$$V_{\m,\mathfrak{b}}: S_{\underline{k}}(\Gamma_0(\m),\chi,\C) \rightarrow S_{\underline{k}}(\Gamma_0(\n),\chi,\C)$$
given by
$$\sum_{\alpha\in\OL_F} c_{\alpha} q^{\alpha} \mapsto \sum_{\alpha\in\OL_F} c_{\alpha} q^{b \alpha}$$
where $\mathfrak{b} = (b)$ and $b\in\OL_F^+$. This map only depends on $b$ up to a scalar which one can easily verify from \eqref{eq_coeff1}. Let $S_{\underline{k}}^{\mathrm{old}}(\Gamma_0(\mathfrak{n}), \chi,\C)$ be the subspace  of $S_{\underline k}(\Gamma_0(\mathfrak{n}),\chi,\C)$ spanned by $V_{\mathfrak{m},\mathfrak{b}}(f)$ for all $f\in S_{\underline k} (\Gamma_0(\mathfrak{m}),\chi,\C)$ and all $(\mathfrak{m},\mathfrak{b})$ with $\mathfrak{m}|\mathfrak{n}$ where $\mathfrak{m}\ne\mathfrak{n}$ and $\mathfrak{b}|(\mathfrak{n}\slash\mathfrak{m})$. The orthogonal complement of $S_{\underline{k}}^{\mathrm{old}}(\Gamma_0(\mathfrak{n}), \C)$, under the Petersson inner product, is the space $S_{\underline{k}}^{\mathrm{new}}(\Gamma_0(\mathfrak{n}), \chi, \C)$; it has a basis of eigenforms which we will refer to as \emph{newforms.}
\medskip

Dembele's algorithm computes the space of newforms $S_{\underline{k}}^{\mathrm{new}}(\Gamma_0(\mathfrak{n}), \chi, \C)$ by using the fact that they are in bijection, via the Jacquet-Langlands correspondence, with a certain space of automorphic forms on a quaternion algebra. We then exploit the fact, special to $\GL(2)$, that the Fourier expansion of a newform can be recovered from its Hecke eigenvalues. Let $\q$ be a non-zero integral prime ideal, and write $\q = (\pi)$ for some totally positive $\pi$. There is a Hecke operator $T_{\q}$ which acts on the space of cusp forms $S_{\underline k}(\Gamma_0(\mathfrak{n}),\chi,\C)$ as defined in \eqref{eq_hecke}. With the identities
$$T_{\q^n} = T_{\q^{n-1}} T_{\q}  -\chi(\mathfrak{q})\pi_1^{k_1-1}\pi_2^{k_2-1} T_{\q^{n-2}},$$
for $(\q,\n) = 1$,
$$T_{\q^n} = T^n_{\q}$$
for $\q |\mathfrak{n}$, and
$$
T_{\mathfrak{rs}}=T_{\mathfrak r}T_\mathfrak{s}
$$
for $(\mathfrak{r},\mathfrak{s})=1$, one can compute the Fourier expansions of the newforms in $S_{\underline k}^{\mathrm{new}}(\Gamma_0(\mathfrak{n}),\chi,\C)$. One can easily calculate the effect of the Hecke operators on formal Fourier expansions indexed over ideals of $\OL_F$ by using \eqref{eq_coeff2}.

\subsection{Eisenstein Series of Weight One}\label{sect:eis}

In \cite{Sh2}, Shimura gives a prescription which attaches to any pair of narrow class characters of $F$ an Eisenstein series of parallel weight $k.$  The Fourier expansions of these Eisenstein series are calculated in \cite{DDP}, and we recall this result here. As we only make use of Eisenstein series of parallel weight $\underline k=[1,1]$ associated to pairs consisting of a trivial and nontrivial character, we include only the details which are relevant to this case.
\medskip

In the classical setting, the Eisenstein series are defined as sums over a lattice, and an analogous construction is used in the case of Hilbert Modular Forms. Let $\psi$ be a totally odd character of the narrow ray class group modulo $\mathfrak{n} $ and let
$$
U=\{u\in\OL_F^\times :\textup{Nm} (u)=1 ,\ u\equiv 1 \mod \mathfrak{n}\}.
$$
For $z\in\H^2$, $s\in\C$ with $\textup{Re}(2s+1)>2$, and $e_F(x)=\exp({2\pi i\cdot\textup{Tr}_{F\slash\Q}}(x))$, define
$$
f(z,s):=C\cdot \frac{1}{\textup{Nm}(\mathfrak{n})}\sum_{\substack{a\in\OL_F,\ b\in\mathfrak{d}^{-1}\\ (a,b)\ \textup{mod}\ U,\ (a,b)\ne(0,0) } } {\left(\frac{1}{(az+b)|az+b|^{2s} }\times\sum_{c\in\OL_F\slash\mathfrak{n}} {\textup{sgn}(c)^{[1,1]}\psi(c)e_F(-bc)} \right)}
$$
where
$$
C:=\frac{\sqrt{d_F}}{[\OL_F^\times : U]\textup{Nm}(\mathfrak{d})(-2\pi i)^2}
$$
and $\textup{sgn}(c)^r:= \textup{sgn}(c_1)^{r_1}\textup{sgn}(c_2)^{r_2}$ and $r=[r_1,r_2]\in(\Z\slash 2\Z)^2$.
\medskip

Observe that the above sum for $f(z,s)$ is over pairs $(a,b)$ of nonzero elements of the product $\OL_f\times \d^{-1}$ modulo the action of $U$ (which is diagonal multiplication) as well as over the representatives $c$ for $\OL_F\slash \mathfrak{n}$.
\medskip

For fixed $z$, $f(z,s)$ has meromorphic continuation in $s$ to the entire complex plane. Set
$$E_{1,\psi}(z):=f(z,0).$$
In \cite{DDP}, the authors compute the Fourier series of the above Eisenstein series, $E_{1,\psi}$. Their result is summarized in the following proposition.

\begin{prop} Let $\mathfrak{n}$ be an integral ideal of $F$ and let $\psi$ be a totally odd character of the narrow ray class group modulo $\mathfrak{n}.$ Then  there exists an element \mbox{$E_{1,\psi} \in M_{[1,1]}(\Gamma_0(\mathfrak{n}),\psi,\C)$} such that $c(\a, E_{1,\psi}) = \sum_{\mathfrak{m} | \a} \psi(\mathfrak{m})$ for all nonzero ideals $\a$ of $\mathcal{O}$ and $c(0, E_{1,\psi}) = \frac{L(\psi,0)}{4}.$ Explicitly,
$$
E_{1,\psi}=\frac{L(\psi,0)}{4}+\sum_{b\in\OL_F^+}{\left(\sum_{\m|(b)} {\psi(\m)} \right)\cdot e_F(bz)}
$$
\end{prop}

\subsection{CM Forms}\label{sect:cmform} While in general spaces of Hilbert modular forms of partial weight one are mysterious, we do have one source to reliably produce such forms; we can obtain them via automorphic induction from certain Gr\"ossencharacters. Specifically, let $K$ be a totally imaginary quadratic extension of $F$ and $\mathbb{A}_K$ be the adeles of $K.$ Consider a Gr\"ossencharacter $$\psi: \textup{GL}_1(K)\backslash \textup{GL}_1(\mathbb{A}_K) \rightarrow \C^{\times}$$ such that the local components of $\psi$ at the infinite places are
$$\psi_{\infty_1}(z) = z^{k-1}\ \ \ \ \ \ \textup{and}\ \ \ \ \ \ \psi_{\infty_2}(z) = |z|_{\infty_2}^{k-1}.$$
Then, by a theorem of Yoshida \cite{Y}, there exists a unique Hilbert modular eigenform $f_{\psi}$ of weight $[k, 1]$ such that the $L$-function of $f_{\psi}$ is equal to the $L$-function of $\psi.$
\medskip

A Hilbert modular eigenform $f$ is said to have CM if its primitive form is equal to $f_{\psi}$ for some character $\psi$. From the equality of $L$-functions, one observes that if $\mathfrak{p}$ is a prime of $F$ which is inert in $K,$ then the normalized Hecke eigenvalue $c(\mathfrak{p}, f_{\psi}) = 0.$ Conversely, this property classifies CM Hilbert modular forms. That is, if $f$ is a Hilbert modular form of level $\mathfrak{c}$ and $K$ is a totally imaginary extension of $F$ such that $c(\mathfrak{p},f) = 0$ for all primes $\mathfrak{p}\nmid \mathfrak{c}$ which are inert in $K,$ the primitive form of $f$ is $f_{\psi}$ for some Gr\"ossencharacter $\psi$ of $K.$ By class field theory, one can restate this fact as follows.

\begin{theorem}\label{theorem:CM} Let $f$ be a Hilbert modular eigenform of level $\mathfrak{c}.$ Then $f$ has CM if and only if there exists a totally odd quadratic Hecke character $\epsilon$ of $F$ of conductor $\mathfrak{f}$ such that $c(\mathfrak{p},f)\epsilon(\mathfrak{p}) = c(\mathfrak{p},f)$ for all $\mathfrak{p}\nmid\mathfrak{cf}.$ In this case, we say $f$ has CM by $\epsilon.$
\end{theorem}

If $f_{\psi}$ is a newform arising from the character $\psi,$ then the level of $f$ is equal to $\Delta_{K/F}N_{K/F}(\mathfrak{f}(\psi))$ where $f(\psi)$ is the conductor of $\psi$. It follows that if $f$ is CM form of level $\Gamma_1(\mathfrak{c}),$ then $f$ has CM by some Hecke character of conductor dividing $\mathfrak{c}.$ There are only finitely many such Hecke characters, and so one can verify by calculating finitely many Hecke eigenvalues of $f$ that $f$ does not have CM.

\subsection{The Algorithm}\label{sect:algorithm}
In this section, we outline the algorithm used to search for non-CM modular forms of weight $[k,1]$.
\medskip

Recall from Section \ref{sect:qexp}, that the nonzero coefficients appearing in the Fourier expansion of a Hilbert modular form are indexed by the totally nonnegative elements of $\mathcal{O}_F.$ Fix a field $H$ and consider the ring of formal Fourier expansions over $H$ (coefficients indexed by the totally nonnegative elements of $\mathcal{O}_F$). For any pair of integers $B := (b_1, b_2)$ there is an ideal of this ring consisting of all formal Fourier series whose Fourier coefficient $c_{\alpha} = 0$ if $|\alpha|_{\infty_1} < b_1$ and $|\alpha|_{\infty_2} < b_2.$ The ring of formal Fourier expansions (over $H$) truncated to bound $B$ is defined to be the quotient of the ring of formal Fourier expansions by this ideal.

\begin{alg}\label{algorithm}
The following is a procedure to search for weight $[k,1]$ modular forms. Which on input $(\underline{k}, \mathfrak{n}, \chi, B)$ consisting of
\begin{enumerate}
\item $\underline{k}=[k,1]$ a pair of odd integers,
\item $\mathfrak{n}$ a square free integral ideal of $F,$
\item $\chi$ a totally odd ray class character of $F$ of conductor dividing $\mathfrak{n}\cdot\infty_1\infty_2$,
\item $B = (b_1, b_2)$ a pair of positive integers,
\end{enumerate}
outputs candidate non-CM weight $\underline{k}$, level $\mathfrak{n}$, character $\chi$ modular forms or finds that none exist.
\end{alg}

\begin{enumerate}
\item\label{step:newforms} Using Demb\'el\'e's algorithm \cite{D} (see section \ref{sect:basis}), compute, for each $\mathfrak{m} |\mathfrak{n},$ a basis for the image of $S^{\new}_{[k+1,2]}(\Gamma_0(\m), F)$ in the ring of formal Fourier expansions over $F$ truncated to bound $B.$

\item\label{step:basis} Using the spaces calculated in step \ref{step:newforms} and following the procedure described in Section \ref{sect:basis}, compute a basis for the image $S_{[k+1,2]}(\Gamma_0(\m), F)$ in the ring of formal Fourier expansions over $F$ truncated to bound $N_{F\slash \Q}(\mathfrak{q})\cdot B$ where $\mathfrak{q}$ is the small prime from Step \ref{step:hecke}.

\item\label{step:divide} Divide each of the truncated Fourier expansions calculated in step \ref{step:basis} by the Fourier expansion for  $E_{1,\chi^{-1}}$ described in Section \ref{sect:eis}. Call the space spanned by the resulting truncated Fourier expansions $V(B).$

\item\label{step:hecke} Choose a small prime $\q$, which was the principal ideal $(2)$ in our case, and compute $T_{\q}f$ for each basis element $f$ of $V(B)$ from the previous step.

\item\label{step:rank} One has now computed two spaces of Fourier expansions, $V(B)$ and $T_{\mathfrak{q}}(V(B))$, each of which are dimension \mbox{$D=\dim S_{[k+1,2]}(\Gamma_0(n))$.} If the the dimension of $V(B)$ is less than $D$, increase $B$. Compute the intersection $V^{(2)}(B):=V(B)\cap T_{\mathfrak{q}}(V(B))$.

\item\label{step:cmform} Compute the dimension of the subspace in $S_{[k+1,2]}(\Gamma_0(\mathfrak{n}),\chi, \C)$ spanned by CM forms using class field theory. Denote this dimension by $h.$

\item\label{step:intersection} If $\dim(V(B)\cap T_{\mathfrak{q}}(V(B)))=h$, then all the forms are CM and the algorithm returns the empty set. Otherwise the algorithm returns $V^{(2)}(B).$

\end{enumerate}

When the algorithm returns a nonempty output one increases the bound $B$ and reruns the algorithm. If the dimension stabilizes at some value greater than the dimension of the space of CM forms after several increases in precision, one has found a candidate for a non-CM weight $\underline{k}$ form.
\medskip

All of our calculations were made for $F = \Q(\sqrt{5}).$ We first used the algorithm to calculate the dimensions of the spaces $M_{[3,1]}(\Gamma_0(\mathfrak n), \chi)$ where $\mathfrak n$ is a square-free ideal of $\OL_F$ and $\chi$ is a totally odd character modulo $\mathfrak n$. We restricted ourselves to the case where $\mathfrak n$ is square-free, because the {\tt magma\rm} package used only worked in this case. Our program searched through all square-free $\mathfrak n$ of norm less than $500$ and quadratic $\chi$, but we did not find any non-CM Hilbert modular forms. (In fact, our calculations show that none exist in the spaces we computed.)

We next used our algorithm to calculate dimensions of $M_{[5,1]}(\Gamma_0(\mathfrak n), \chi)$ for all square-free ideals $\mathfrak n$ of norm less than 300. The only candidate space our algorithm found is described below in Section \ref{sect:nonCMform}. In all other spaces of modular forms, our algorithm found that all forms were CM.

\section{A non CM form}\label{sect:nonCMform}
Let $F = \Q(\sqrt{5}).$ We order the infinite places of $F$ such that $|\sqrt{5}|_{\infty_1} > 0.$  The ray class group of conductor
$(7) \cdot \infty_1 \infty_2$ is isomorphic to $\Z/6\Z$. Let $\chi$ be the order $6$ character such that
$\chi(2) =  \frac{-1 + \sqrt{-3}}{2}$. The character  $\chi$ is totally odd.
\medskip

\begin{theorem} \label{theorem:exists} The space of cusp forms
$S_{[5,1]}(\Gamma_0(14),\chi,\C)$ is $2$-dimensional and has a basis with coefficients in $H$. This space has a basis over $L=H(\sqrt{-19})$ consisting of two conjugate eigenforms, neither of which admit complex multiplication.
\end{theorem}

\begin{proof}
For $n$ a positive integer, we define
$$b(n) := \left(\frac{5n - \sqrt{5}n}{2}, \frac{5n + \sqrt{5}n}{2}\right).$$ Applying Algorithm $1$ with input $(\underline{k}, \mathfrak{n}, \chi, B) = ([5,1],14\mathcal{O}, \chi, B)$ with $B= b(24),b(26)$ and $b(28),$ respectively, one finds that for each value  $V^{(2)}(B)$ is two dimensional. \mbox{Table \ref{table_coeff}} lists the initial normalized Fourier coefficients of one of the truncated forms in $V^{(2)}(B).$ Let $f \in S_{[6,2]}(\Gamma_0(14),1, H)/E_{1, \chi^{-1}}$ be a meromorphic modular form whose Fourier expansion truncated to $b(28)$ is found in Table $\ref{table_coeff}.$ We show $f \in S_{[5,1]}(\Gamma_0(14),\chi, H),$ by showing $f^3 \in S_{[15,3]}(\Gamma_0(14),\chi^3, H).$ This is done in two steps.
\begin{enumerate}
\item First we show the map taking a form in $S_{[18,6]}(\Gamma_0(14),H)$ to its Fourier expansion truncated to bound $b(28)$ is an injection.
\item Next we find a form $g \in S_{[15,3]}(\Gamma_0(14),\chi^3,H)$ such that the Fourier expansions of $g$ and $f$ are equivalent when truncated to bound $b(28).$
\end{enumerate} Noting that $(f^3 - g)E^3_{1,\chi} \in S_{[18,6]}(\Gamma_0(14),H),$ it follows from $(1)$ and $(2)$ that $f^3$ and $g$ are equal.

The proofs of facts $(1)$ and $(2)$ are both computational. Using the {\tt magma\rm} package, one computes that the space of cusp forms $S_{[18,6]}(\Gamma_0(14),H)$ has dimension $356.$ Then one computes explicitly the Fourier expansions truncated to bound $b(28)$ for a basis of $S_{[18,6]}(\Gamma_0(14),H)$ and shows that the resulting set of truncated formal Fourier series span a space of the same dimension. This proves $(1).$

To prove $(2)$, one must construct an element $S_{[15,3]}(\Gamma_0(14),\chi^3)$ with a desired property. Unfortunately, the creation of spaces of Hilbert modular forms with nontrivial nebentypus and the computation of their Fourier expansions has not yet been implemented in the {\tt magma\rm} package. To skirt this issue, we instead use the {\tt magma\rm} package to compute the truncated to bound $b(56)$ Fourier expansions of the $56$ dimensional space $S_{[14,2]}(\Gamma_0(14),\chi^3, H).$ One then obtains the Fourier expansions for the forms in the subspace
$$
E_{1,\chi^3}.S_{[14,2]}(\Gamma_0(14),\chi^3, H) + T_2(E_{1,\chi^3}.S_{[14,2]}(\Gamma_0(14),\chi^3, H)) \subseteq S_{[15,3]}(\Gamma_0(14),\chi^3)
$$
truncated to bound $b(28),$ in which, following a calculation in linear algebra, one finds a form $g$ as desired in $(2).$
It follows $S_{[5,1]}(\Gamma_0(14),\chi,\C)$ is $2$-dimensional and has a basis with elements in $H.$

We now demonstrate the second claim of the proposition: that $S_{[5,1]}(\Gamma_0(14),\chi,\C)$ has a basis over $L=H(\sqrt{-19})$ consisting of two conjugate eigenforms, neither of which admit complex multiplication.  Utilizing Algorithm \ref{algorithm}, one computes that \mbox{$V^{(2)}([5,1], 7\mathcal{O}, \chi, b(28)) = 0$} and hence  $S_{[5,1]}(\Gamma_0(14),\chi,\C) = S^{\mathrm{new}}_{[5,1]}(\Gamma_0(14),\chi,\C).$ It follows $S_{[5,1]}(\Gamma_0(14),\chi,\C)$ has a basis over $\C$ of simultaneous eigenforms for Hecke algebra. As $S_{[5,1]}(\Gamma_0(14),\chi,\C)$ has a basis defined over $H$ and is two dimensional, these eigenforms have as a field of definition either $H$ or a quadratic extension of $H.$ Calculating the characteristic polynomial of $T_5$ on $S_{[5,1]}(\Gamma_0(14),\chi,\C),$ we obtain that the field of definition is $H(\sqrt{-19}).$

Finally, we see that neither of the forms in $S_{[5,1]}(\Gamma_0(14),\chi,\C)$ are CM. If this were not the case, both forms of $S_{[5,1]}(\Gamma_0(14),\chi,\C)$ would have CM by a quadratic character of conductor $14.$ The unique such character is $\chi^3$. However, one observes that $\chi^3(\frac{7 + \sqrt{5}}{2}) = -1$ and the $\frac{7 + \sqrt{5}}{2}$ normalized Hecke eigenvalue does not vanish for either eigenform in $S_{[5,1]}(\Gamma_0(14),\chi,\C)$. \end{proof}

\begin{remark} The Galois group $\Gal(L/\Q) = (\Z/2\Z)^3$ acts on the Fourier expansion as follows. The element with fixed field $H$ permutes the two eigenforms. The element with fixed field $\Q(\sqrt{5},\sqrt{-19})$ sends the eigenform to an eigenform in $S_{[5,1]}(\Gamma_0(14),\chi^{-1},\C)$, where $\chi^{-1}$ is the conjugate of $\chi$. The element with fixed field $\Q(\sqrt{-3},\sqrt{-19})$ sends the eigenform to a form in $S_{[1,5]}(\Gamma_0(14),\chi,\C)$.
\end{remark}

See Table \ref{table_coeff} for the normalized coefficients $c(\p)$ for various prime ideals $\p=(\pi)$ of small norm for one of the two normalized eigenforms in $S_{[5,1]}(\Gamma_0(14),\chi,\C)$. If $c(\pi)$ is a coefficient in the Fourier expansion of our eigenform for a prime $\pi$, then the normalized coefficient is $c(\p)=c(\pi)\overline\pi^2$ as seen in \eqref{eq_coeff2}. The normalized coefficient does not depend on the choice of totally positive generator $\pi$ for the ideal $\p=(\pi)$.

\begin{table}[!ht]
\caption{Table of Normalized Coefficients of Eigenform in $S_{[5,1]}(\Gamma_0(14),\chi) $}\label{table_coeff}
\centering
\begin{tabular}{|c|c|c|}
\hline
$\pi$ & $N(\pi)$ &  $c(\p), \ \p = (\pi)$ \\
\hline
$2$ & $4$ &  $\displaystyle{-4+4\sqrt{-3}}$\\
$\frac{5 + \sqrt{5}}{2}$ & $5$  & $\displaystyle{\frac{-45+15\sqrt{-3}+15\sqrt{-19}-15\sqrt{57}}{4}}$\\
$3$ & $9$  & $\displaystyle{-18-18\sqrt{-3}-9\sqrt{-19}\left(\frac{3-\sqrt{-3}}{2} \right)}$\\
$\frac{7 + \sqrt{5}}{2}$ & $11$ & $\displaystyle{\frac{-87+87\sqrt{-3}+36\sqrt{5}-36\sqrt{-15}+63\sqrt{-19}-21\sqrt{57} +24\sqrt{-95}-8\sqrt{285}}{4}}$\\
$\frac{9 + \sqrt{5}}{2}$ & $19$ & $\displaystyle{\frac{-456+152\sqrt{-3} +171\sqrt{5}-57\sqrt{-15} +66\sqrt{-19} -66\sqrt{57} -39\sqrt{-95} +39\sqrt{285}}{4}}$\\
$\frac{11 + \sqrt{5}}{2}$ & $29$ & $\displaystyle{-162+\frac{417}{2}\sqrt{5}+66\sqrt{57}+\frac{17}{2}\sqrt{285}}$\\
$\frac{13 + \sqrt{5}}{2}$ & $41$ & $\displaystyle{\left(49+12\sqrt{5} \right)\cdot\left(\frac{9\sqrt{-3}+15\sqrt{-19}}{2} \right)}$\\
$7$ & $49$ & $\displaystyle{\frac{-1715+1715\sqrt{-3}+1029\sqrt{-19}+1029\sqrt{57}}{4}}$\\
\hline
\end{tabular}
\end{table}

\begin{remark} \emph{We checked that for $N(\p) < 1000$ and $\gcd(N(\p),14)=1$ the Satake parameters of $\pi$ satisfy the Ramanujan Conjecture.
Equivalently, the Hecke eigenvalues satisfy the bounds $|c(\p)|_{\infty_1} \le 2 p^{2}$ and $|c(\p)|_{\infty_2} \le 2 p^{2}$. The Ramanujan conjecture
would follow from Deligne's proof of the Riemann hypothesis if one knew that~$\pi$ was \emph{motivic}, however, the construction of the associated Galois
representations proceeds via congruences.}
\end{remark}

\section{Acknowledgements}
The authors would like to extend special thanks to Frank Calegari for introducing this problem, his excellent advising, and his extensive comments on preliminary drafts of this paper. The authors would also like to thank Kevin Buzzard, Lassina Demb\'el\'e, and James Newton for their comments on a preliminary version of this paper, and Don Blasius, Kevin Buzzard, and Fred Diamond for their historical remarks.

\section{Funding}
The first author was supported in part by National Science Foundation Grant DMS-1404620. The second author was supported in part by National Science Foundation Grant DMS-1404620 and by an National Science Foundation Graduate Research Fellowship under Grant No. DGE-1324585.

\end{document}